\documentclass[11pt,a4paper]{article}

\usepackage{epsf,epsfig,amsfonts,amsgen,amsmath,amstext,amsbsy,amsopn,amsthm
}
\usepackage{ebezier,eepic}
\usepackage{color}

\newtheorem{theorem}{Theorem}

\newtheorem{false statement}{False statement}

\theoremstyle{definition}

\newtheorem{claim}{Claim}

\newtheorem{conjecture}{Conjecture}

\renewcommand{\mod}{{\rm mod}}

\newcounter{mathitem}
  {\begin{list}{{$(\roman{mathitem})$}}{
   \setcounter{mathitem}{0}
   \usecounter{mathitem}
   \setlength{\topsep}{0pt plus 2pt minus 0pt}
   \setlength{\parskip}{0pt plus 2pt minus 0pt}
   \setlength{\partopsep}{0pt plus 2pt minus 0pt}
   \setlength{\parsep}{0pt plus 2pt minus 0pt}
   \setlength{\leftmargin}{35pt}
   \setlength{\itemsep}{0pt plus 2pt minus 0pt}}}
  {\end{list}}

\date{\dateline{July 27, 2013}{}}
\begin{document}

\title{\bf\Large Rainbow $C_4$'s and Directed $C_4$'s: the Bipartite Case Study}

\date{}

\author{Bo Ning$^{a}$\thanks{E-mail address: ningbo\_math84@mail.nwpu.edu.cn (BO NING).} and Jun Ge$^{b}$\thanks{Corresponding author. E-mail address: mathsgejun@163.com (JUN GE).}\\
\small $^{a}$Department of Applied Mathematics, School of Science,\\
\small  Northwestern Polytechnical University, Xi'an, Shaanxi 710072, P.R.~China\\[2mm]
\small $^{b}$School of Mathematical Sciences, \\
\small Xiamen University, Xiamen, Fujian 361005, P.R. China}
\maketitle
\begin{abstract}
In this paper we obtain a new sufficient condition for the existence
of directed cycles of length 4 in oriented bipartite graphs. As a
corollary, a conjecture of H. Li is confirmed. As an application,
a sufficient condition for the existence of rainbow cycles of length 4
in bipartite edge-colored graphs is obtained.


\end{abstract}

\section{Introduction}
For terminology and notation not defined here, we refer to \cite{BM}. Let $G=(V,E)$ be a simple graph.
An \emph{edge-coloring} of $G$ is a mapping $C: E\rightarrow \mathbb N^{+}$,
where $\mathbb N^{+}$ is the set of positive integers. We
call $G^c$ an \emph{edge-colored graph} (or briefly, a \emph{colored graph})
if $G$ is assigned an edge-coloring $C$. Let $v$ be a vertex of
$G^c$. The \emph{color degree} of $v$ in $G^c$, denoted by $d^{c}_{G}(v)$
(or briefly, $d^{c}(v)$), is the number of colors of the edges incident
to $v$. A \emph{color neighborhood set} $N^c (v)$ of $v$ is a subset of $N(v)$, the
neighborhood of $v$, such that the colors of edges between $v$ and $N^c (v)$ are pairwise distinct.
Let $H$ be a subgraph of $G$, then $C(H)=\{C(e):e\in E(H)\}$ is called the \emph{color set} of $H$.

A subgraph of a colored graph is called \emph{rainbow}
(sometimes called \emph{heterochromatic}, or \emph{colorful})
if all edges of it have distinct colors. The existence of
rainbow subgraphs has been studied for a long time. A
problem of rainbow Hamilton cycles in colored complete
graphs was mentioned by Erd\"{o}s, Ne\v{s}et\v{r}il
and R\"{o}dl \cite{ENR}, and later studied by Hahn and Thomassen \cite{HT},
Frieze and Reed \cite{FR} and Albert, Frieze and Reed \cite{AFR}, respectively.
Rainbow matchings were studied by Wang and H. Li \cite{WL},
Lesaulnier et al. \cite{LSWW}, and Kostochka and Yancey \cite{KY}.
Chen and X. Li \cite{CL_0,CL_5} studied the existence of long
rainbow paths. A recent article on strong rainbow connection
can be found in \cite{LS}. For a survey on the study of
rainbow subgraphs in colored graphs, we refer to \cite{KL}.

In particular, rainbow short cycles have received much attention.
Broersma et al. \cite{BLWZ} studied the existence of rainbow $C_3$'s
and $C_4$'s under color neighborhood union condition. Later,
H. Li and Wang \cite{LW} obtained two results on the existence of
rainbow $C_3$'s and $C_4$'s under colored degree condition.

\begin{theorem}[Li and Wang \cite{LW}]\label{t1}
Let $G^c$ be a colored graph of order $n\geq 3$. If $d^{c}(v)\geq
(4\sqrt{7}/7-1)n+3-4\sqrt{7}/7$ for each $v\in V(G)$,
then $G^c$ has either a rainbow $C_3$ or a rainbow $C_4$.
\end{theorem}

\begin{theorem}[Li and Wang \cite{LW}]\label{t2}
Let $G^c$ be a colored graph of order $n\geq 3$. If $d^{c}(v)\geq (\sqrt{7}+1)n/6$
for each $v\in V(G)$, then $G^c$ has a rainbow $C_3$.
\end{theorem}

H. Li and Wang \cite{LW} conjectured that every colored graph $G^c$
of order $n\geq 3$ has a rainbow $C_3$ if $d^{c}(v)\geq (n+1)/2$
for each $v\in V(G)$. This conjecture was proved by H. Li \cite{L} and
stronger results were proved by B. Li et al. \cite{LNXZ} with different methods as follows.

\begin{theorem}[Li \cite{L}]\label{t3}
Let $G^c$ be a colored graph of order $n\geq 3$. If $d^{c}(v)\geq (n+1)/2$
for each $v\in V(G)$, then $G^c$ has a rainbow $C_3$.
\end{theorem}

\begin{theorem}[Li, Ning, Xu and Zhang \cite{LNXZ}]\label{t4}
Let $G^c$ be a colored graph of order $n\geq 3$. If
$\sum\limits_{v\in V(G)}d^c(v)\geq n(n+1)/2$, then
$G^c$ has a rainbow $C_3$.
\end{theorem}

\begin{theorem}[Li, Ning, Xu and Zhang \cite{LNXZ}]\label{t5}
Let $G^c$ be a colored graph of order $n\geq 3$. If $d^{c}(v)\geq n/2$
for each $v\in V(G)$, then $G^c$ has a rainbow $C_3$ or
$G\in \{K_{n/2,n/2}, K_4-e, K_4 \}$.
\end{theorem}

The existence of rainbow $C_4$'s in special colored graphs has
also been studied. Wang et al. \cite{WLZL} obtained a result on
the existence of rainbow $C_4$'s in triangle-free colored graphs.
Recently, H. Li \cite{L} got a result on the existence of rainbow
$C_4$'s in balanced bipartite colored graphs.

\begin{theorem}[Wang, Li, Zhu and Liu \cite{WLZL}]\label{t6}
Let $G^c$ be a triangle-free colored graph of order $n\geq 9$. If
$d^{c}(v)\geq (3-\sqrt{5})n/2+1$ for each $v\in V(G)$, then
$G^c$ has a rainbow $C_4$.
\end{theorem}

\begin{theorem}[Li \cite{L}]\label{t7}
Let $G^c$ be a balanced bipartite colored graph of order $2n$ with
bipartition $(A,B)$. If $d^c(v)> 3n/5+1$ for each $v\in A\cup B$,
then $G^c$ has a rainbow $C_4$.
\end{theorem}

While in \cite{L}, H. Li made a tiny error in the proof of
Theorem \ref{t7}. Notice that $K_{3,3}$ is 3-edge-colorable, and a proper
3-edge-coloring of $K_{3,3}$ satisfies the condition of Theorem \ref{t7},
but it has no rainbow $C_4$ since there are only 3 colors.
We point out that, in order to correct it, the condition
$d^c(u)>3n/5+1$ should be changed into $d^c(u)>(3n+8)/5$.

Now we turn to finite simple oriented graphs, i.e., finite graphs
without multiple edges and loops in which each edge is replaced
by exactly one arc. Let $D[A,B]$ be an oriented bipartite graph with
bipartition $(A,B)$. When there is no ambiguity, we use $D$ instead
of $D[A,B]$. For $A_1\subseteq A$ and $B_1\subseteq B$, we denote by
$A_{D}(A_1,B_1)$ the set of arcs from $A_1$ to $B_1$ in $D[A,B]$.

The study of rainbow cycles in colored graphs is largely related to
the study of oriented cycles in digraphs. For a wonderful example,
see the introduction of \cite{L}. In particular, motivated by the study
of short rainbow cycles in colored graphs, H. Li \cite{L} proposed the
following nice conjecture and proved for balanced oriented bipartite graphs.

\begin{conjecture}[Li \cite{L}]\label{c1}
Let $D$ be an oriented bipartite graph with bipartition $(A,B)$.
If $d^{+}(u)>|B|/3$ for each $u\in A$ and $d^{+}(v)>|A|/3$
 for each $v\in B$, then $D$ has a directed $C_4$.
\end{conjecture}

\begin{theorem}[Li \cite{L}]\label{t8}
Let $D$ be a balanced oriented bipartite graph with bipartition $(A,B)$,
where $|A|=|B|=n$. If $d^{+}(v)>n/3$ for each $v\in A\cup B$,
then $D$ has a directed $C_4$.
\end{theorem}

We state a construction from \cite{WLZL} to show that if Conjecture
\ref{c1} holds, then it would be almost the best possible. Let
$m$ and $n$ be two positive integers divisible by 3. Let
$|M_0|=|M_1|=|M_2|=m/3$ and $|N_0|=|N_1|=|N_2|=n/3$.
We construct an oriented bipartite graph with bipartition $(M, N)$,
where $M=M_0\cup M_1\cup M_2$ and $N=N_0\cup N_1\cup N_2$, by
creating all possible arcs from $M_i$ to $N_i$, and from $N_i$ to
$M_{i+1}$, $i=0,1,2$ (modulo 3). In the rest parts,
we use $D^*(m,n)$ to denote the construction above.

The first purpose of this paper is to confirm Conjecture \ref{c1}.
In fact, we prove a stronger result as follows.

\begin{theorem}\label{t9}
Let $D$ be an oriented bipartite graph with bipartition $(A,B)$,
where $|A|=m\geq 2$ and $|B|=n\geq 2$. If $d^{+}(u)\geq n/3$
for each $u\in A$ and $d^{+}(v)\geq m/3$  for each $v\in B$,
then either $D$ has a directed $C_4$ or $D=D^*(m,n)$.
\end{theorem}

By using Theorem \ref{t9}, we can extend Theorem \ref{t7} as follows.

\begin{theorem}\label{t10}
Let $G^c$ be a bipartite colored graph with bipartition $(A,B)$.
If $d^c(u)\geq (3|B|+8)/5$ for each $u\in A$ and
$d^c(v)\geq (3|A|+8)/5$ for each $v\in B$,
then $G^c$ has a rainbow $C_4$.
\end{theorem}

\section{Proofs}

{\bf {Proof of Theorem \ref{t9}}}

Let $\mathcal{D}(m,n)$ be the family of digraphs consisting of those oriented
bipartite graphs with bipartition $(A,B)$ which satisfies the condition
of Theorem \ref{t9}, where $m=|A|$ and $n=|B|$.

First we claim that it is sufficient to prove for those $m$ and $n$
which are both multiples of 3. Suppose Theorem \ref{t9} holds
for $\mathcal{D}(m,n)$ with $m\equiv n\equiv 0 ~(\mod ~3)$.
For any $D\in \mathcal{D}(m',n')$, where $m'$ and $n'$ are not
both multiples of 3, let $s_1=3\lceil\frac{m'}{3}\rceil-m'$,
$A^{*}=\{u_1,\ldots,u_{s_1}\}$ and $A'=A\cup A^{*}$. Let
$s_2=3\lceil\frac{n'}{3}\rceil-n'$, $B^{*}=\{v_1,\ldots,v_{s_2}\}$
and $B'=B\cup B^{*}$. Now we construct
a new oriented bipartite graph $D'$ with bipartition $(A',B')$,
where $A(D')=A(D)\cup \{u'v,v'u: u\in A,u'\in A^{*},v\in B, v'\in B^{*}\}$.
Notice that $d^{+}_{D'}(u)\geq \lceil\frac{|B|}{3}\rceil=\frac{|B'|}{3}$
for each $u\in A$ and $d^{+}_{D'}(u')=n'>\frac{|B'|}{3}$ for each $u'\in A^{*}$.
Similarly, $d^{+}_{D'}(v)\geq \lceil\frac{|A|}{3}\rceil=\frac{|A'|}{3}$
for each $v\in B$ and $d^{+}_{D'}(v')=m'>\frac{|A'|}{3}$ for each $v'\in B^{*}$.
It follows that $D'\in \mathcal{D}(3\lceil\frac{m'}{3}\rceil, 3\lceil\frac{n'}{3}\rceil)$.
Hence $D'$ has a directed $C_4$ or $D'=D^*(3\lceil\frac{m'}{3}\rceil, 3\lceil\frac{n'}{3}\rceil)$.
Since $A^{*}$ and $B^{*}$ are not both empty sets and the vertices in
$A^{*}$ and $B^{*}$ only have outdegrees,
$D'\neq D^*(3\lceil\frac{m'}{3}\rceil, 3\lceil\frac{n'}{3}\rceil)$, and moreover,
the directed $C_4$ in $D'$ is also in $D$. The proof of our claim is complete.

Now assume $D\in\mathcal{D}(m, n)$, where $m=3m_1$, $n=3n_1$ and
$m_1,n_1$ are two positive integers. Let $D_1$ be a spanning
subdigraph of $D$ satisfying $d_{D_1}^{+}(u)=n_1$ for each $u\in A$
and $d_{D_1}^{+}(v)=m_1$ for each $v\in B$. Suppose $D$ has no
directed $C_4$, obviously, $D_1$ also has no directed $C_4$.
Let $u_0$ be a vertex with maximum indegree $k_1$ among $A$, and
$v_0$ be a vertex with maximum indegree
$k_2$ among $B$. Let $B_1=N_{D_1}^{-}(u_0)$,
$B_2=N_{D_1}^{+}(u_0)$, $A_3=N_{D_1}^{+}(B_2)$ and
$B_3=N_{D_1}^{+}(A_3)-B_2$, where $|B_1|=k_1$,
$|B_2|=n_1$. Since $D_1$ has no directed $C_4$,
we have $N_{D_1}^{+}(A_3)\cap B_1=\emptyset$. Since $k_2$ is the maximum indegree
of all vertices in $B$, we get
\begin{align}
|B_3|k_2\geq |N_{D_1}^{-}(B_3)|=|A_{D_1}(A_3,B_3)|.
\end{align}
Since $D_1$ has no directed $C_4$, there is no arc from $A_3$ to $B_1$, which implies
all arcs starting from $A_3$ have heads in $B_2\cup B_3$. Hence
$|A_{D_1}(A_3,B_3)|=|A_3|n_1-|A_{D_1}(A_3,B_2)|$. Since
\begin{align}
|A_{D_1}(A_3,B_2)|\leq |A_3||B_2|-|A_{D_1}(B_2,A_3)|,
\end{align}
we obtain
\begin{align}
|A_{D_1}(A_3,B_3)|\geq |A_3|n_1-(|A_3||B_2|-|A_{D_1}(B_2,A_3)|)
=|A_{D_1}(B_2,A_3)|=m_1n_1.
\end{align}
Together with (1) and (3), we obtain $|B_3|\geq \frac{m_1n_1}{k_2}$. Therefore
\begin{align}
3n_1=|B|\geq |B_1|+|B_2|+|B_3|\geq k_1+n_1+\frac{m_1n_1}{k_2}\geq 3\sqrt[3]{\frac{k_1m_1n_1^2}{k_2}}.
\end{align}
It follows that $k_2n_1\geq k_1m_1$. By symmetry, we also have $k_1m_1\geq k_2n_1$.
Thus, $k_1m_1=k_2n_1$ and all the inequalities (1)-(4) are actually equalities.
These facts imply $|B_1|=|B_2|=|B_3|=n_1=k_1$, $m_1=k_2$,
$|A_{D_1}(B_2,A_3)|=|A_{D_1}(A_3,B_3)|=m_1n_1$,
and all vertices in $A$ have indegree $n_1$ in $D_1$,
all vertices in $B$ have indegree $m_1$ in $D_1$.

Next we show $|A_3|=m_1$. First, choose a vertex $v'\in B_2$, $d^+(v')=m_1$,
so $|A_3|\geq m_1$ by the definition of $A_3$. Assume that $|A_3|>m_1$.
Since the inequality (2) becomes equality, the underlying graph of
$D_1[A_3,B_2]$ is a complete bipartite graph. Hence there exists a
vertex, say $u'\in A_3$, such that $u'v'\in D_1$. Since $u'\in A_3$,
there exist a vertex $v''\in B_2$, such that $v''u'\in D_1$ by the
choice of $A_3$. Note that $d^+_{D[B_2,A_3]}(v')=d^+_{D[B_2,A_3]}(v'')=m_1$
and $u'\notin N^+(v')$. It follows that there exists a vertex,
say $u''\in A_3$, such that $v'u''\in D_1$ and $v''u''\notin D_1$.
Since the underlying graph of $D_1[B_2,A_3]$ is a complete bipartite graph,
$u''v''\in D_1$. Now $C=u'v'u''v''u'$ is a directed $C_4$ in $D_1$,
a contradiction. Hence $|A_3|=m_1$.

Now let $A_1=N_{D_1}^{+}(B_3)$ and $A_2=N_{D_1}^{-}(B_2)$.
Note that $|\{vu: v\in B_2, u\in A_3\}|=m_1n_1=|A_{D_1}(B_2,A_3)|$. It follows
that $A_{D_1}(B_2, A_3)=\{vu: v\in B_2, u\in A_3\}$. Similarly,
$A_{D_1}(A_3, B_3)=\{uv: u\in A_3, v\in B_3\}$. So there
is no arc with tail in $A_3$ and head in $B_2$, or arc with tail in
$B_3$ and head in $A_3$, follows $A_1\cap A_3=A_2\cap A_3=\emptyset$.
Since $D_1$ has no directed $C_4$, $A_1\cap A_2=\emptyset$.

Since $d_{D_1}^+(u)=d_{D_1}^-(u)=n_1$ for each $u\in A$ and
$d_{D_1}^+(v)=d_{D_1}^-(v)=m_1$ for each $v\in B$, we obtain
$|A_1|\geq m_1$ by its definition, and
$n_1|A_2|\geq |A_{D_1}(A_2, B_2)|=\sum_{v\in B_2} d_{D_1}^- (v)=m_1n_1$,
follows that $|A_2|\geq m_1$. Since $|A_1|+|A_2|+|A_3|=3m_1$
and $A_1, A_2, A_3$ are pairwise disjoint, $|A_1|=|A_2|=|A_3|=m_1$.
Now apparently, $A_{D_1}(A_2,B_2)=\{uv:u\in A_2,v\in B_2\}$ and
$A_{D_1}(B_3,A_1)=\{vu:v\in B_3,u\in A_1\}$. Since $A_1$ and $A_2$ are
disjoint, $A_{D_1}(A_1,B_2)=\emptyset$. Furthermore, $A_{D_1}(A_1,B_3)=\emptyset$,
hence $N^+(A_1)\subset B_1$. $\sum_{v\in A_1}d^+(v)=m_1n_1$
implies $A_{D_1}(A_1,B_1)=\{uv:u\in A_1,v\in B_1\}$. Similarly, $A_{D_1}(B_1,A_2)=
\{vu: v\in B_1,u\in A_2\}$. Therefore $D_1=D^{*}(m,n)$.
If there is any arc in $D$ but not in $D_1$, then obviously, there
would be a directed $C_4$ in $D$, a contradiction. Thus $D=D_1=D^{*}(m,n)$.
The proof is complete. {\hfill$\Box$}

{\bf {Proof of Theorem \ref{t10}}}

First note that the color degree condition implies that $|A|\geq 4$ and $|B|\geq 4$.

Suppose not. Let $G^c$ be a colored graph which satisfies the condition of
Theorem \ref{t10} but has no rainbow $C_4$. Set $|A|=n_1$ and $|B|=n_2$.

Choose an edge $e=xy\in E(G^c)$ such that $C(xy)=c_0$. Let
$N^c(x)=\{y,y_1,y_2,\ldots,y_{r-1}\}$ and
$N^c(y)=\{x,x_1,x_2,\ldots,x_{s-1}\}$.
Since $d^c(x)\geq \frac{3n_2+8}{5}$ and $d^c(y)\geq \frac{3n_1+8}{5}$, we can
set $r=\lceil\frac{3n_2+8}{5}\rceil$ and $s=\lceil\frac{3n_1+8}{5}\rceil$.
Let $A_1=\{x_1,x_2,\ldots,x_{s-1}\}$ and $B_1=\{y_1,y_2,\ldots,y_{r-1}\}$.
Note that $G^c[A_1,B_1]$ is also a bipartite colored graph.

The following claim can be deduced immediately from the definition of
color neighborhood set and the assumption that $G^c$ has no rainbow $C_4$.

\begin{claim}\label{cl1}
For any edge $x_iy_j$ and $C(xy_j)\neq C(yx_i)$, where $1\leq i\leq s-1$
and $1\leq j\leq r-1$, we have $C(x_iy_j)\in \{C(xy),C(xy_j),C(yx_i)\}$.
\end{claim}

Now we construct an oriented bipartite graph $D=D[A_1,B_1]$ as follows. For any edge
$x_iy_j\in E(G[A_1,B_1])$, such that $C(x_iy_j)\neq C(xy)$ and $C(xy_j)\neq C(yx_i)$,
then $C(x_iy_j)=C(xy_j)$ or $C(x_iy_j)=C(yx_i)$ by Claim \ref{cl1}.
If $C(x_iy_j)=C(xy_j)$, we define an arc $x_i y_j$ in $D$, and if $C(x_iy_j)= C(yx_i)$,
we define an arc $y_j x_i$ in $D$. Let $G'=G[D]$ be the underlying graph of
$D$.

In the following, for convenience, when we mention the color of an
arc in $D$, we mean the color of the corresponding edge in $E(G')$.

\begin{claim}\label{cl2}
There is no directed $C_4$ in $D$.
\end{claim}
\begin{proof}
Suppose $Q=x_iy_jx_py_qx_i$ is a directed $C_4$ in $D$. By the
definition of $D$, we have $C(x_iy_j)=C(xy_j)$, $C(y_jx_p)=C(yx_p)$,
$C(x_py_q)=C(xy_q)$ and $C(y_qx_i)=C(yx_i)$. The existence of arcs
$x_iy_j$ and $y_jx_p$ implies $C(xy_j)\neq C(yx_i)$
and $C(xy_j)\neq C(yx_p)$, hence $C(x_iy_j)\neq C(y_qx_i)$ and
$C(x_iy_j)\neq C(y_jx_p)$. We have $C(xy_j)\neq C(xy_q)$ from the
definition of $B_1$, hence $C(x_iy_j)\neq C(x_py_q)$. So $C(x_iy_j)$
is different from the colors of all other three edges in the cycle $Q$
in $G^c$. Similarly we can prove that all edges in $Q$ receive distinct
colors in $G^c$, and therefore, $Q$ is a rainbow $C_4$ in $G^c$, a contradiction.
\end{proof}

\begin{claim}\label{cl3}
$D\neq D^*(|A_1|,|B_1|)$.
\end{claim}
\begin{proof}
Assume that $D=D^*(|A_1|,|B_1|)$. Without loss of generality,
set $A_1=X_1\cup X_2\cup X_3$ and $B_1=Y_1\cup Y_2\cup Y_3$,
where $|X_1|=|X_2|=|X_3|$ and $|Y_1|=|Y_2|=|Y_3|$. Since
$s-1\equiv 0~(\mod ~3)$ and $r-1\equiv 0~(\mod ~3)$, we have
$\lceil\frac{3n_i+3}{5}\rceil\equiv 0~(\mod ~3)$, $i=1, 2$.
It follows that $n_i\equiv 3\ \text{or} \ 4\ (\mod \ 5)$, $i=1, 2$.

First, we claim that $D\neq D^*(3,3)$. Suppose not. Then
$\lceil\frac{3n_1+3}{5}\rceil=\lceil\frac{3n_2+3}{5}\rceil=3$.
Hence $n_1=n_2=4$. In this case, we may suppose that $X_i=\{x_i\}$
and $Y_i=\{y_i\}$ for $i=1,2,3$. The existence of arc $x_1y_1$ in
$D$ implies that $C(x_1y_1)=C(xy_1)$. Hence $d^c(y_1)\leq 3< \frac{3|A|+8}{5}$,
a contradiction.

Let $v_0$ be an arbitrary vertex of $D$. Without loss of generality, assume
$v_0\in Y_1$. If $n_1=5k+\lambda$, $\lambda=3\ \text{or}\ 4$, then $|X_1|=|X_2|=|X_3|=k+1$.
From the definition of $D$, we know that each edge in $\{uv_0:u\in X_1\}$ has
the same color $C(v_0x)$, and there are $k+1$ different colors in
$\{v_0u: u\in X_2\}$. Since $|A\backslash (A_1\cup \{x\})|=5k+\lambda-(3k+3)-1=2k+\lambda-4$,
there are at most $2k+\lambda-4$ different colors in the edge set
$\{v_0u:u\in A\backslash (A_1\cup \{x\})\}$. These facts mean that
there are at most $3k+\lambda-2$ different colors in the edge set
$\{v_0u:u\in A\backslash X_3\}$. Since $d^c(v)\geq 3k+\lambda$,
there exists an edge $v_0u_0\in \{v_0u:u\in X_3\}$, such that
$v_0u_0$ has a new color and $v_0u_0$ is not in $E(G')$.

Next we show that any directed path of length 3 in $D=D^*(|A_1|,|B_1|)$
is rainbow. Without loss of generality, we choose a directed path
$uv'u'v$, where $u'\in X_1$, $v\in Y_1$, $u\in X_3$ and $v'\in Y_3$.
By the construction of $D$, $C(u'v)=C(xv)$, $C(v'u')=C(yu')$ and
$C(uv')=C(xv')$. Since $u'v$ exists in $D$, $C(xv)\neq C(yu')$.
It follows that $C(u'v)\neq C(v'u')$. Similarly, we have
$C(v'u')\neq C(uv')$. Since $C(u'v)=C(xv)$ and $C(uv')=C(xv')$
and $C(xv)\neq C(xv')$ by the choice of $N^c(x)$, we have
$C(u'v)\neq C(uv')$.

Now we fix a vertex $v_0\in Y_1$. By the analysis above, there
exists an edge $v_0u_0\in \{v_0u\in E(G): u\in X_3\}$, which is
not in $E(G')$ and satisfies $C(v_0u_0)\notin \{C(uv_0\textcolor{red}{)}:u\in X_1\}$.
Since $D=D^*(|A_1|,|B_1|)\neq D^*(3,3)$, there are at least two
arcs in $A_D(Y_3,X_1)$ with distinct colors, and we can choose
one of them, say $v_0'u_0'\in A_D(Y_3,X_1)$, such that $C(v_0'u_0')\neq C(v_0u_0)$. By
the analysis before, there is an edge $u_0v_0''\in E(G^c)$, where
$v_0''\in Y_1$, such that it is not in $E(G')$ and satisfies $C(u_0v_0'')\neq C(u_0v_0')$.
Now we will show that $v_0''=v_0$. First, the deletion of
$u_0v_0$ means $C(u_0v_0)=C(xy)$ or $C(xv_0)=C(yu_0)$.
If $C(u_0v_0)=C(xy)$, then the existence of arcs $u_0'v_0$, $u_0'v_0'$ and $u_0v_0'$
in $D$ implies $C(xy)\neq C(u_0'v_0)$, $C(xy)\neq C(u_0'v_0')$
and $C(xy)\neq C(u_0v_0')$. Since colors of arcs in $\{u_0'v_0,v_0'u_0',u_0v_0'\}$
are pairwise distinct, $u_0'v_0'u_0v_0u_0'$ is a rainbow $C_4$ in $G^c$,
a contradiction. Hence $C(xv_0)=C(yu_0)$. Similarly, we may obtain $C(xv_0'')=C(yu_0)$
from the deletion of $u_0v_0''$ when constructing $D$. It follows that $C(xv_0)=C(xv_0'')$.
Now we can get $v_0=v_0''$ directly from the definition of $N^c(x)$.
From the analysis above, colors of arcs in $\{u_0'v_0,v_0'u_0',u_0v_0'\}$
are pairwise distinct, and $C(v_0u_0)$ is different from all of the three.
Therefore $u_0'v_0u_0v_0'u_0'$ is also a rainbow $C_4$ in $G^c$, a contradiction.
\end{proof}

By Claims \ref{cl2} and \ref{cl3}, $D$ has no directed $C_4$
and $D\neq D^*(|A_1|,|B_1|)$. By Theorem \ref{t9},
there exists a vertex, without loss of generality, say $y_j\in B_1$, such that
$d^{+}_{D}(y_j)<\frac{|A_1|}{3}$. By the construction of $D$,
we know there are less than $\frac{|A_1|}{3}+1$ different
colors in $C(E_{G'}(y_j, A_1))$. For any edge $e$ adjacent
to $y_j$ which is in $E(G[A_1, B_1])\backslash E(G')$, $C(e)=C(xy)$
or there exists an edge $yx_i$ such that $C(yx_i)=C(xy_j)$ in $G^c$, and in this case,
$e=x_iy_j$ and $C(e)$ maybe missed in $C(E_{G'}(y_j, A_1))$. This implies
that there are at most three colors in $C(xy_j)\cup C(E_{G}(y_j, A_1))\backslash C(E_{G'}(y_j, A_1))$.
However, $C(xy_j)$ is also in $C(E_{G'}(y_j,A_1))$. Hence there are less than $\frac{|A_1|}{3}+3$ different
colors in $C(E_{G}(y_j,\{x\}\cup A_1))$. It follows that there are more than
$d^c(y_j)-3-\frac{|A_1|}{3}\geq s-3-\frac{s-1}{3}$ different
colors in the color set $\{y_jx':x'\in A\backslash (A_1\cup\{x\})\}$. Hence $y_j$ has more
than $s-3-\frac{s-1}{3}$ different neighbors in $A\backslash (A_1\cup\{x\})$.
Now we have
$$n_1=|A|\geq |A_1|+|\{x\}|+|N(y_j)\backslash (A_1\cup\{x\})|>(s-1)+1+s-3-\frac{s-1}{3}=\frac{5s-8}{3}\geq n_1,$$
a contradiction.

The proof is complete.
{\hfill$\Box$}

\section*{Acknowledgement}
Bo Ning is supported by NSFC (No. 11271300) and the Doctorate Foundation
of Northwestern Polytechnical University (cx201326). Jun Ge is supported
by NSFC (No. 11171279 and No. 11271307). The authors are grateful to editors
and anonymous referees for helpful comments on an earlier version of this article.

\end{document}